\documentclass{amsart}
\usepackage[english]{babel}
\usepackage{amsrefs}
\usepackage{amssymb,amsmath,amsthm,amsfonts,epsfig,graphicx,psfrag,color}
\usepackage[utf8]{inputenc}
\usepackage{xcolor}
\usepackage{hyperref}
\hypersetup{
    colorlinks,
    citecolor=black,
    filecolor=blue,
    linkcolor=black,
    urlcolor=blue
}
\graphicspath{{figuras/}}

\DeclareMathOperator{\sign}{sign}

\theoremstyle{plain}
\newtheorem{thm}{Theorem}

\newtheorem{lem}[thm]{Lemma}
\newtheorem{prop}[thm]{Proposition}

\theoremstyle{definition}
\newtheorem{rmk}[thm]{Remark}

\theoremstyle{remark}

\DeclareMathOperator{\sg}{sg}

\newcommand{\R}{\mathbb R}

\newcommand{\C}{\mathbb C}
\newcommand{\Z}{\mathbb Z}

\newcommand{\E}{\mathcal E}
\renewcommand{\P}{\mathcal P}

\renewcommand{\epsilon}{\varepsilon}

\newcommand{\spin}{\mathcal S}

\begin{document}

\author{Alfonso Artigue}
\title{Relativistic one-dimensional billiards}
\date{\today}
\email{aartigue@litoralnorte.udelar.edu.uy}
\address{Departamento de Matem\'atica y Estad\'\i stica del Litoral,
Centro Universitario Regional Litoral Norte,
Universidad de la Rep\'ublica,
Florida 1065, Paysand\'u, Uruguay.}

\begin{abstract}
In this article we study the dynamics of one-dimensional relativistic billiards containing particles with positive and negative energy.
We study configurations with two identical positive masses and symmetric positions with
two massless particles between them of negative energy and symmetric positions.
We show that such systems 
have finitely many collisions in any finite time interval.
This is due to a phenomenon we call \textit{tachyonic collision}, 
which occur at small scales and produce changes in the sign of the energy of individual particles.
We also show that depending on the initial parameters the solutions can be bounded with certain periodicity or unbounded while obeying an inverse square law at large distances.
\end{abstract}
\maketitle


\section{Introduction}

This article is a \textit{relativistic continuation} of \cite{Art19} where
we studied the dynamics of a system of particles in the real line, performing elastic collisions and without external forces;
we assumed Newtonian dynamics and allowed the masses to have an arbitrary sign.
A particular situation resulted of interest: two particles of positive mass
with a particle of negative mass between them. 
We call the particle in the middle a \textit{toy graviton} 
because successive bounces accelerate the positive mass particles toward each other, simulating an attracting force.
We toke the limit as its mass tends to zero, with constant kinetic energy, 
obtaining a \textit{limit system} that turned to be equivalent to replace the toy graviton by a potential $U=-k/R^2$, where $k$ is a constant and $R$ the distance between the positive masses. This was already known for positive masses \cite{BTT2007}.
In the present paper we explore these ideas in the context of special relativity.

Let us describe the contents of this article while stating the results we obtained.
In \S\ref{secFreeP} we describe free particles with mass and energy of arbitrary sign and in \S\ref{secDefDyn} we define the dynamics.
We consider elastic collision between such particles.
In \S\ref{secTachCol1} we define \textit{tachyonic collision} as two colliding particles making a system with more momentum than energy, and we show that they are associated to the change in the sign of the energy of the particles. An example and some interpretations are given in \S\ref{secZECol}.

In \S\ref{secMirrorSystem} we consider \textit{mirror systems}, which have just four particles with symmetric positions with respect to the origin.
This simplification allows to reduce the dynamics to a one-dimensional map in \S\ref{secReductionOneDim}.
We assume positions $x_1\leq x_2\leq x_3\leq x_4$ with $x_1=-x_4$ and $x_2=-x_3$,
masses $m_1=m_4=m$ and the second and third particles are massless with energy $E_2=E_3=\E$.
We define $\Delta=\E^2-m^2$.
For this systems we prove that
if $\Delta\geq 0$ then:
\begin{itemize}
 \item The first and final particles goes to infinity in the future and in the past, while the energy of the middle particles tend to zero, in the future and in the past. In particular these systems do not collapse and experience a \textit{bounce}, see Proposition \ref{propEGtoZero}.
\item In Proposition \ref{propSquareLaw} we show that under the previous assumptions, in limit times (past and future)
the energy of the middle particles is proportional to the inverse of the distance of the first and last particles. This motivates to think about the middle particles as some kind of \textit{gravitons} whose energy plays the role of the gravitational energy.
\end{itemize}
In \S\ref{secSmallSP} we consider the case $\Delta<0$ and show that its solutions are
bounded and have no collapse. Also, some of them are periodic and we calculate its period.

Section \S\ref{secTachCol} is devoted to study tachyonic collisions.
In Theorem \ref{thmCondTach} we show that for $\Delta\geq 0$ the solutions have at most two tachyonic collision, while for $\Delta<0$ they appear infinitely many times.
In Proposition \ref{propEscalaTach} it is shown that tachyonic collisions are present at \textit{small scales}.
In \S\ref{secEstimation} we give an estimation for real particles.


\section{Free particle}
\label{secFreeP}
We start with one particle. Its position will be denoted as $x(t)\in\R$ for all $t$ in a real interval (mainly, it will be $t\in\R$).
As a free particle its velocity is a constant $v=\frac{dx}{dt}(t)$.
The particle has constant energy $E\in\R$, $E\neq 0$, and its momentum is defined as $P=Ev$.
The \textit{squared mass} is $\mu=E^2-P^2$.

From these conditions it follows that for a massless particle, $\mu=0$, we have
$E^2=P^2$ and its velocity is $v=P/E=\pm 1$.
That is, we choose units so that the speed of light equals 1.
For $\mu<0$ we have $|v|>1$ and this particle is usually called a \textit{tachyon}.
For $\mu>0$, if $m$ is its mass then $\mu=m^2$.
For $\mu\geq 0$ the particle is a \textit{bradyon}.

\begin{rmk}
With these definitions we avoid the square roots that appears in special relativity textbooks.
We can easily obtain $E^2=\mu/(1-v^2)$, but to write the usual formula $E=m/\sqrt{1-v^2}$ is a problem for us because, even assuming $1-v^2>0$, the signs of energy and mass are forced to coincide, as long as the chosen square root is the positive one.
But, as we will see in \S \ref{secZECol}, we can have a particle changing the sign of its energy during a collision.
\end{rmk}

We derive a simple formula for the velocity that will be useful later.
Consider the quantities
\[
 \left\{
\begin{array}{l}
\sigma=E+P\\
\rho=E-P.
\end{array}
\right.
\]

\begin{rmk}
We see that $\sigma\rho=\mu$, $2E=\sigma+\mu/\sigma$
and $2P=\sigma-\mu/\sigma$. Thus
\begin{equation}
 \label{ecuVel}
 v=\frac{P}{E}=\frac{\sigma-\mu/\sigma}{\sigma+\mu/\sigma}
 =\frac{\sigma^2-\mu}{\sigma^2+\mu}.
\end{equation}
\end{rmk}

\begin{rmk}[Interpretation of negative energy]
It is important to remark that all we will do only depends on the squared masses and not on the mass.
However, we find useful to have a way of thinking about positive mass with negative energy.
Besides, this phenomenon is unavoidable in the dynamics we will study, as we will see in \S \ref{secTachCol}.
To choose a square root $m$ of $\mu$ we need extra structure: the particle will carry a watch.
The hand of the watch can spin clockwise or counterclockwise.
The spin velocity is defined as
$$\spin=\frac mE$$
provided that $m$ is any chosen square root of $\mu$, $m^2=\mu$.
Since $E^2=\mu/(1-v^2)$ we have $|E| \sqrt{1-v^2}=|m|$ and $|\spin|=\sqrt{1-v^2}$.
Also, $\sign(\spin)=\sign(m/E)$.
Therefore, for $m\neq 0$ we have
\[
 E=\frac{m}{\spin}=\sign(m/E)\frac m{\sqrt{1-v^2}}.
\]
A particle whose watch spins counterclockwise ($\spin<0$) is to be interpreted as traveling backward in time.
With full generality (any real value of $m$) we can write $m=\spin E$, where vanishing mass is equivalent to a stopped watch.
\end{rmk}

\section{Definition of the dynamics}
\label{secDefDyn}
In this section we consider one-dimensional dynamics of $N\geq 3$ particles with positions $x_1\leq x_2\leq \dots\leq x_N$.
The time evolution is defined by the rules:
\begin{enumerate}
 \item between collisions the particles are free: have constant velocity $v_i=P_i/E_i$,
 \item the collisions are elastic: the energy and momentum are preserved,
 \item there are no multiple collisions at the same time and at the same point.
\end{enumerate}

\begin{rmk}
We have to rule out multiple simultaneous collisions at the same point because, in general, there is no satisfactory way to define the result. This is because we can approximate it by ordering the collisions in different ways and obtaining different results, which implies a discontinuity in the dynamics. However, this is a non-generic event.
We remark that more than one collision at the same time but at different places (a description of the events that depend on the reference frame) are allowed and will be used in \S \ref{secMirrorSystem}.
\end{rmk}

\subsection{Collisions}
The previous rules are sufficient to \textit{resolve} the collisions, as we will see in Lemma \ref{lemaSudoku}.
If we use primes to denote the final energy and momentum of the particles $i,j$ after a collision, the conservation of energy and momentum gives us:
$E_i + E_j = E'_i + E'_j$ and $P_i+P_j = P'_i+P'_j$.
To solve these equations we follow \cite{Lap86} and
define
\[
 \left\{
\begin{array}{l}
\sigma_i=E_i+P_i\\
\rho_i=E_i-P_i
\end{array}
\right.
\]
for each particle and the corresponding \textit{primed} quantities after the collision.
Jointly with $E_*^2-P_*^2=\mu_*$ ($*=i,j$) we have the system of equations:
\begin{equation}
 \label{ecuSudoku}
 \text{before collision}\left\{
 \begin{array}{l}
 \sigma_i+\sigma_j=s\\
 \rho_i+\rho_j=r\\
 \sigma_i\rho_i=\mu_i\\
 \sigma_j\rho_j=\mu_j
 \end{array}
\right.
\Rightarrow
 \text{after collision}\left\{
 \begin{array}{l}
 \sigma'_i+\sigma'_j=s\\
 \rho'_i+\rho'_j=r\\
 \sigma'_i\rho'_i=\mu_i\\
 \sigma'_j\rho'_j=\mu_j
 \end{array}
\right.
\end{equation}
where $s=\sigma_i+\sigma_j$ and $r=\rho_i+\rho_j$ are constants of the collision (energy and momentum conservation).

\begin{rmk}[Collision condition]
Given two particles $i,j$ the condition for a collision is $v_i-v_j\neq 0$.
The fact that the collision occurs in the future or in the past depends on the positions, this condition just guarantees
a collision (in positive or negative time).
As the velocity of a particle $*$ is
$
v_*=\frac{P_*}{E_*}=\frac{\sigma_*-\rho_*}{\sigma_*+\rho_*}$
the collision condition is
\begin{equation}
 \label{ecuCondiciónChoque}
 \sigma_i\rho_j-\sigma_j\rho_i \neq 0 \Leftrightarrow v_i-v_j\neq 0.
\end{equation}
\end{rmk}

\begin{lem}
\label{lemaSudoku}
If $sr\neq 0$ then the system of equations \eqref{ecuSudoku} has just one other solution:
 \[
\left\{
\begin{array}{l}
 \displaystyle \sigma'_i=\rho_i\frac{s}{r}=\sigma_i-\frac{\sigma_i\rho_j-\sigma_j\rho_i}
 {r}\\
 \displaystyle \sigma'_j=\rho_j\frac{s}{r}=\sigma_j+\frac{\sigma_i\rho_j-\sigma_j\rho_i}
 {r}\\
 \end{array}
\right.
\left\{
\begin{array}{l}
 \displaystyle \rho'_i=\sigma_i\frac{r}{s}=\rho_i+\frac{\sigma_i\rho_j-\sigma_j\rho_i}
 {s}\\
 \displaystyle \rho'_j=\sigma_j\frac{r}{s}=\rho_j-\frac{\sigma_i\rho_j-\sigma_j\rho_i}
 {s}\\
 \end{array}
\right.
\]
which is different from the original if and only if
the collision condition \eqref{ecuCondiciónChoque} holds.
\end{lem}
The proof of Lemma \ref{lemaSudoku} is left to the reader.
Some remarks are in order.

\begin{rmk}
 The condition $sr=0$ has probability zero, thus the dynamics is generically well defined.
\end{rmk}

\begin{rmk}
From the definitions we have:
 \[
 \begin{array}{rl}
  sr & =(\sigma_i+\sigma_j)(\rho_i+\rho_j)= (E_i+P_i + E_j+P_j)(E_i-P_i + E_j-P_j)\\
     & =(\E+\P)(\E-\P)=\E^2-\P^2,  \\
 \end{array}
 \]
where $\E=E_i+E_j$ and $\P=P_i+P_j$.
The quantity $sr=\E^2-\P^2$ represents \textit{the rest mass of this system}.
\end{rmk}

\begin{rmk}
Notice that
the collision of two particles with the same mass is
always solved (in a classical or a relativistic framework) by exchanging energy and momentum, or equivalently we can think that they just interact by exchanging their labels.
\end{rmk}

\begin{rmk}
\label{rmkChoqueMasasIguales}
If we consider classical mechanics and solve a collision we find a denominator in the formula of the final velocity as $m_i+m_j$.
Therefore, $m$ cannot collide with $-m$.
It is remarkable that
from the relativistic viewpoint we are considering, this collision is always allowed.
Indeed, by direct inspection of Equations \eqref{ecuSudoku} we see that
if $\mu_i=\mu_j$ then
 $\sigma'_i=\sigma_j$,
 $\sigma'_j=\sigma_i$,
 $\rho'_i=\rho_j$,
 $\rho'_j=\rho_i$
 is the solution of the collision.
 Considering the previous remark we can say that a particle of mass $m$ does not interact with particles of mass $m$ or $-m$.
\end{rmk}

\subsection{Tachyonic collisions}
\label{secTachCol1}
In the classical case the kinetic energy is $mv^2/2$ and its sign only depends on the sign of the mass, thus the sign of the energy is preserved after collisions.
In the relativistic case the sign of the energy can change during a collision, see an example in \S\ref{secZECol}.
A \textit{tachyonic collision} is a collision with $sr=\E^2-\P^2<0$.

\begin{prop}
\label{propTachColi}
 If $\mu_i\geq 0$ then the energy of the particle $i$ changes its sign during a collision
if and only if it is a tachyonic collision.
\end{prop}

\begin{proof}
 From definitions and Lemma \ref{lemaSudoku} we have that
 \[\left\{
 \begin{array}{l}
 \sigma_i\rho_i=\mu_i,\\
2E_i=\sigma_i+\rho_i,\\
2E'_i=\sigma'_i+\rho'_i=\rho_i \frac sr+\sigma_i \frac rs=\frac{\rho_i s^2+\sigma_ir^2}{rs}.\\
 \end{array}
\right.
\]
Since $\mu_i=\sigma_i\rho_i\geq 0$ we have that $(\sigma_i+\rho_i)(\rho_i s^2+\sigma_ir^2)>0$.
Thus, $\sg(E_iE_i')=\sg(rs)$.
\end{proof}

%

\subsection{Example of a zero-energy tachyonic collision}
\label{secZECol}
Let us consider a curious collision.
Suppose two particles, the particle 1 is at rest with mass $m=1$ and the particle 2 is massless, has energy $E_2=-1$, velocity $-1$ and momentum $P_2=1$.
In this way the system has vanishing energy $\E=0$ and momentum $\P=1$.
In particular, as $\E^2-\P^2=-1<0$, their collision is tachyonic.
Lemma \ref{lemaSudoku} gives us the energies and momenta of both particles after the collision: $E_1'=-1$, $P_1'=0$, $E_2'=1$ and $P_2'=1$.
It is remarkable that particle 1 continues at rest after the collision.
Its energy turns negative since the collision is tachyonic.

\section{Mirror systems}
\label{secMirrorSystem}
In this section we extend the study of
\cite{Art19}. To simplify the analysis we consider a symmetric system.

\subsection{General properties of mirror systems}

Suppose $N=4$ particles under the following conditions:
$x_1=-x_4$, $x_2=-x_3$, $\mu_1=\mu_4=\mu>0$.
The \textit{middle particles} are massless: $\mu_2,\mu_3=0$.
The collisions between the particles $1-2$ are simultaneous with $3-4$.
Due to the symmetry the particles $2-3$ have collisions only at $x=0$
and there they exchange energy and momentum as explained in Remark \ref{rmkChoqueMasasIguales}.

Suppose a collision $1-2$.
As particle 1 is on the left and particle 2 is massless we have
$v_2=-1$.
In this case, $E_2=-P_2$ before the collision.
Thus, $\sigma_2=0$ and $\rho_2=2 E_2$.
Next, particle 2 collides with particle 3 exchanging their momenta while both have the same energy, again, due to symmetry.
After this second collision we have $v_2'=-1$, $E_2'=-P_2'$, $\sigma_2'=0$ and $\rho_2'=2E_2'$.

Applying Lemma \ref{lemaSudoku} we obtain
\begin{equation}
 \label{ecuE2Espejo}
\left\{
\begin{array}{l}
E'_2=\frac 12\sigma_2'=\frac 12\rho_2\frac sr=E_2\frac{\sigma_1+\sigma_2}{\rho_1+\rho_2}=\frac{E_2\sigma_1}{2E_2+\mu/\sigma_1}\\
 \sigma'_1=\rho_1\frac sr=\frac{\mu}{\sigma_1}\frac sr=\frac{\mu}{2E_2+\mu/\sigma_1}
\end{array}
\right.
\end{equation}

This reduces the dynamics to the study of the iterations of the two-dimensional map $(\sigma_1,E_2)\mapsto (\sigma_1',E_2')$, but it can be simplified even more.

\subsubsection{Reduction to a one-parameter system}
\label{secReductionOneDim}
As $\sigma_1=E_1+P_1$ and $\rho_1=E_1-P_1$ we have
$\sigma_1+\rho_1=2E_1=2E-2E_2$, where $\E=E_1+E_2$ (a constant due to symmetry).
Thus, $\sigma_1+\mu/\sigma_1=2\E-2E_2$,
\begin{equation}
 \label{ecuSimpliSigma}
 2E_2+\mu/\sigma_1=2\E-\sigma_1
\text{ and }
 \sigma'_1=\frac{\mu}{2\E-\sigma_1}.
\end{equation}
Thus, the dynamics is reduced to the one-dimensional map
\begin{equation}
\label{ecuReduc1Dim}
f(\sigma)=\frac{\mu}{2\E-\sigma}
\end{equation}
for any $\sigma\in\R, \sigma\neq 2\E$.
Define
\[
 \Delta=\E^2-\mu.
\]
As we will see, the sign of $\Delta$ characterizes the dynamics. 
For instance, 
the fixed points of $f$ are determined by the quadratic equation
$\sigma^2-2\E\sigma+\mu=0$ whose discriminant is $\Delta$. 
Thus, in what follows $\Delta$ will be called as \textit{discriminant} of the system.

\subsubsection{A constant of motion}
In what follows we derive a relation that will be useful in the next sections.
Consider $\tau_n$ the time elapsed between the collisions  of the particle 1 (with particle 2);
and $v_1^n$ the velocity of particle 1 between such collision.
Let $x_j^n$ denote the position of particle $j$ when it has the $n$-th collision.
Let $E_j^n, P_j^n, v_j^n$ be the energy, momentum and velocity of the particle $j$ between its $n$ and $n+1$ collision, respectively.

\begin{lem}
\label{lemSuperRecurrencia}
The number
$
 {x_1^n E_2^n}/{\sigma_1^n}
$
is independent of 
$n\in\Z$.
\end{lem}

\begin{proof}
On the one hand, as $c=1$ we have
$\tau_n=-x_1^n-x_1^{n+1}$ (the negative signs are due to the fact that the position
$x_1^n$ is negative).
On the other hand, $v_1^n=(x_1^{n+1}-x_1^n)/\tau_n$.
Therefore, $$v_1^n=\frac{x_1^{n+1}-x_1^n}{-x_1^n-x_1^{n+1}}\text{ and }
\frac{x_1^{n+1}}{x_1^n}=\frac{1-v_1^n}{1+v_1^n}=\frac\mu{(\sigma_1^n)^2}.$$
For the last equality we have applied Equation \eqref{ecuVdeSigma}.
From
\eqref{ecuE2Espejo} and \eqref{ecuSimpliSigma} we have
$$\frac{E_2^{n+1}}{E_2^n}=
\frac{\sigma_1^n}{2E_2^n+\mu/\sigma_1^n}=\frac{\sigma_1^n}{2\E-\sigma_1^n}.
$$
Then
\[
 \frac{x_1^{n+1} E_2^{n+1}}{x_1^nE_2^n}=\frac\mu{\sigma_1^n(2\E-\sigma_1^n)}=\frac{\sigma_1^{n+1}}{\sigma_1^n}.\qedhere
\]
\end{proof}


\subsection{Non-negative discriminant}
We will assume that $\Delta=\E^2-\mu\geq 0$.
Considering the map $f$ from \S \ref{secReductionOneDim}
we need to know that the fixed points $\sigma_*$ satisfies
  \begin{equation}
  \label{ecuFixedPtEq}
\sigma_*^2    -2\E\sigma_*+\mu=0.
  \end{equation}
If in addition $\E>0$
the fixed points are labeled as:
 \[
  \left\{
  \begin{array}{l}
    \sigma_{at}=\E-\sqrt{\Delta}\\
    \sigma_{re}=\E+\sqrt{\Delta}.
  \end{array}
  \right.
 \]
and the derivative at the fixed points is
$$\frac{df}{d\sigma}(\sigma_{at})=\frac{\mu}{(2\E-\sigma_{at})^2}=\frac{\mu}{2\E^2-\mu+2\E\sqrt{\Delta}}<1$$
Analogously,
$$\frac{df}{d\sigma}(\sigma_{re})=\frac{\mu}{2\E^2-\mu-2\E\sqrt{\Delta}}>1$$
thus, $\sigma_{at}$ is an attractor and $\sigma_{re}$ is a repeller.
If $\E<0$ then we define  (changing signs)
\begin{equation}
 \label{ecuAtRepENeg}
  \left\{
  \begin{array}{l}
    \sigma_{at}=\E+\sqrt{\Delta}\\
    \sigma_{re}=\E-\sqrt{\Delta}.
  \end{array}
  \right.
\end{equation}
and again $\sigma_{at}$ is an attractor and $\sigma_{re}$ is a repeller.
For $\Delta=0$ there is just one fixed point $\sigma_*=\E$ which is neither attractor nor repeller.
In what follows we 
study the time limit of the systems we are considering.

\begin{prop}
\label{propEGtoZero}
If $\Delta\geq 0$ then $x_1^n\to-\infty$, $x_4^n\to+\infty$ as $n\to\pm\infty$ and the energy of particles 2 and 3 goes to zero as time goes to infinity.
\end{prop}

\begin{proof}
From our previous analysis we have that for any initial condition $\sigma_1\neq \sigma_{re}$ we have $f^n(\sigma_1)\to \sigma_{at}$ as $n\to+\infty$.
Analogously, if $\sigma_1\neq \sigma_{at}$ we have $f^n(\sigma_1)\to \sigma_{re}$ as $n\to-\infty$.
Also,
\begin{equation}
\label{ecuE1(Sigma1)}
2E_1=\sigma_1+\rho_1=\sigma_1+\frac{\mu}{\sigma_1},
\end{equation}
$\E=E_1+E_2$ and
\[
 2E_2=2\E-\sigma_1-\frac{\mu}{\sigma_1}.
\]
Therefore, taking limit and applying
Equation \eqref{ecuFixedPtEq}
we conclude
\[
 \lim_{n\to+\infty} 2E_2^{n}=2\E-\sigma_{at}-\frac{\mu}{\sigma_{at}}
 =0.
\]
Analogously, the energy of the second particle tends to zero as $n\to-\infty$.
The symmetry of the system allows us to obtain the same conclusions for the particle 3.

From Equation \eqref{ecuVel} we have
\begin{equation}
 \label{ecuVdeSigma}
v_1=\frac{\sigma_1^2-\mu}{\sigma_1^2+\mu}.
\end{equation}
Arguing as before we can take the limits $n\to\pm\infty$.
If $\Delta>0$ and $\E>0$ we obtain
\begin{equation}
 \label{ecuVelLim}
\begin{array}{l}
 \displaystyle\lim_{n\to+\infty}v_1^{n}=\frac{\sigma_{at}^2-\mu}{\sigma_{at}^2+\mu}=
 \frac{\Delta-\E\sqrt{\Delta}}{\E^2-\E\sqrt{\Delta}}=
 -\frac{\sqrt{\Delta}}{\E}<0,\\

 \displaystyle\lim_{n\to-\infty}v_1^{n}=\frac{\sigma_{re}^2-\mu}{\sigma_{re}^2+\mu}=
 \frac{\Delta+\E\sqrt{\Delta}}{\E^2+\E\sqrt{\Delta}}=
 \frac{\sqrt{\Delta}}{\E}>0.\\
\end{array}
\end{equation}
This implies that in the past the particle 1 has positive velocity, while in the future it is negative.
That is, $x_1^n\to-\infty$ as $n\to\pm\infty$.
The result for $x_4^n$ follows by symmetry.
The case $\Delta>0$ and $\E<0$ is analogous.

If $\Delta=0$ then the fixed point is $\sigma_*=\E$ and $\E^2=\mu$.
If $\E>0$ then for $0<\sigma<\E$ we have $\sigma^2<\E^2=\mu$.
Thus, $\lim_{n\to+\infty}v_1^{n}=0^-$.
Analogously, $\lim_{n\to-\infty}v_1^{n}=0^+$.
This implies that the limit velocities have the same sign as for $\Delta>0$.
The case $\E<0$ with $\Delta=0$ is follows in a similar way.
\end{proof}

\begin{rmk}
From Proposition \ref{propEGtoZero} we conclude that the limit energy of the \textit{toy gravitons} is zero. If their negative energy were to be equated with gravitational energy of the form
$U=k-Gm_1m_2/R$ then it means that the constant $k$ should be zero.
\end{rmk}


Following \cite{Art19} we say that a solution \textit{collapses} if it has infinitely many
collisions in finite time.

\begin{rmk}
From Proposition \ref{propEGtoZero} we see that, independently of how big is $E_1$ and how small is $E_2$, the system does not collapse.
\end{rmk}

%
%
%


In the next result we show that in the limit the energy of the middle particles is inversely proportional to the position of the other particles. This translates into an inverse square law for the force.
\begin{prop}
\label{propSquareLaw}
If $\Delta\geq 0$ then $E_2(t) x_1(t)$ converges as $t\to\pm\infty$.
If $\Delta=0$ both limits coincide.
\end{prop}

\begin{proof}
We know that $x_1(t)\to -\infty$ as $t\to+\infty$ and
for an initial condition $\sigma_1\neq \sigma_{re}$ we have
that
$f^n(\sigma_1)\to \sigma_{at}\text{ as }n\to+\infty$.
Applying Lemma \ref{lemSuperRecurrencia} (which holds for any value of $\Delta$) we conclude that
\[
 x_1^nE_2^n\to x_1^0E_2^0 \sigma_{at}/\sigma_1^0\text{ as } n\to+\infty,
\]
\[
 x_1^nE_2^n\to x_1^0E_2^0 \sigma_{re}/\sigma_1^0\text{ as } n\to-\infty.
\]
For the particular case $\Delta=0$ we have that $\sigma_{at}=\sigma_{re}$ and both limits coincide.
\end{proof}


\subsection{Negative discriminant}
\label{secSmallSP}
In this section we will study the case $\Delta<0$
for the mirror system of 4 particles as before.
The map $f$ of Equation \eqref{ecuReduc1Dim}
$f(\sigma)=\frac{\mu}{2\E-\sigma}$ has two fixed points $\sigma_{at}$ and
$\sigma_{re}$, which are complex since $\Delta<0$.
Let $\lambda=\frac{\sigma_{at}}{\sigma_{re}}$.
Since $|\lambda|=1$ we can write $\lambda=e^{i\theta}$.

Let us remark some facts that will be used in the next proof. 
First notice that $f$ can be regarded as a Möbius transformation on the complex plane. 
We will show that $f$ is conjugate to the rotation $g(z)=\lambda z$ as follows.
Let $h\colon\C\to\C$ be defined as
$$h(\sigma)=\frac{-\sigma_{at}+\sigma}{\sigma_{re}-\sigma}=z\text{ with inverse } h^{-1}(z)=\frac{\sigma_{at}+z\sigma_{re}}{1+z}.$$
It satisfies: $h$ transforms the real axis into the unit circle, $h(\sigma_{at})=0$, $h(\sigma_{re})=\infty$ and $h(\E)=1$.
To prove
$g(z)=hfh^{-1}(z)$ notice that 
$\sigma_{at}\sigma_{re}=\mu$ and $\sigma_{at}+\sigma_{re}=2\E$. Then 
\[
\begin{array}{rl}
h(f(\sigma)) & = h\left(\frac\mu{2\E-\sigma}\right) = \frac{-\sigma_{at}+\frac\mu{2\E-\sigma}}{\sigma_{re}-\frac\mu{2\E-\sigma}}=\frac{\sigma_{at}}{\sigma_{re}}\left(\frac{-\sigma_{at}+\sigma}{\sigma_{re}-\sigma}\right)=\lambda h(\sigma).
\end{array}
\]
This shows that $f$ is conjugate to a rotation if the discriminant is negative. However, this conjugacy is true for any sign of the discriminant.

Next we will show that for a dense set of the parameters $\mu$ and $E$ the solutions are periodic. In fact, we give a quite complete description of the solutions.

\begin{thm}
For a mirror system of 4 particles, if $\Delta<0$ then the solutions are bounded.
 If in addition $\theta=\frac ab 2\pi$, with $a,b$ coprime positive integers,
 then the solutions are periodic with period $$T=\frac{2kb\mu}{\mu-\E^2}$$
where $k=\frac{x_1^0E_2^0}{\sigma_1^0}$.
\end{thm}

\begin{proof}
From Lemma \ref{lemSuperRecurrencia} we known that
$x_1^n=k\sigma_1^n/E_2^n$.
By Equation \eqref{ecuSimpliSigma} we know that
$2E_2^n=2\E-\sigma_1^n-\mu/\sigma_1^n$.
Thus
\begin{equation}
\label{ecuXnRecu}
 x_1^n=\frac{2k\sigma_1^n}{2\E-\sigma_1^n-\mu/\sigma_1^n}.
\end{equation}
The solutions are bounded because
the map $L\colon\R\to\R$ given by
\begin{equation}
 \label{ecuFuncionL}
L(\sigma)=
\frac{2k\sigma}{2\E-\sigma-\mu/\sigma}=
\frac{2k\sigma^2}{2\E\sigma-\sigma^2-\mu}=
\frac{-2k\sigma^2}{(\sigma-\E)^2-\Delta}
\end{equation}
is bounded since $\Delta<0$ and it has finite limit as $\sigma\to\pm\infty$.

Now we use that $f$ is conjugate to the rotation $g$ of angle $\theta=\frac ab 2\pi$ in the circle $|z|=1$. Thus,
$g^b(z_0)=z_0$ if $|z_0|=1$.
That is, the \textit{discrete period} is $b$ (the solution closes the cycle after $b$ collisions). To calculate the time elapsed in this cycle we
consider the function $L$ of Equation \eqref{ecuFuncionL} given by
\begin{equation}
L(\sigma)=\frac{-2k\sigma^2}{(\sigma-\E)^2-\Delta}.
\end{equation}
As the speed of light is $c=1$ we have that
\[
 -\frac12T=\sum_{j=1}^{j=b} L(f^j(\sigma_0))
\]
where $\sigma_0\in\R$ is any initial condition for the iteration of $f$.
If we let $z_0=h(\sigma_0)$ and $z_j=g^j(z_0)=\lambda^jz_0$ then
$$
f^j(\sigma_0)
=f^j(h^{-1}(z_0))
=h^{-1}(g^j(z_0))
=h^{-1}(z_j)
=\displaystyle \frac{\sigma_{at}+z_j\sigma_{re}}{1+z_j}
$$
and we can perform the following calculations:
\[
 L(f^j(\sigma_0))=\frac{-2k(\frac{\sigma_{at}+z_j\sigma_{re}}{1+z_j})^2}{(\frac{\sigma_{at}+z_j\sigma_{re}}{1+z_j}-\E)^2-\Delta}
\]
\[
\begin{array}{rl}
 \frac{\sigma_{at}+z_j\sigma_{re}}{1+z_j}-\E & =
 \frac{(\sigma_{at}-\E)+z_j(\sigma_{re}-\E)}{1+z_j}=
 \frac{-\sqrt{\Delta}+z_j\sqrt{\Delta}}{1+z_j}\\
 &=\frac{-1+z_j}{1+z_j}\sqrt{\Delta}
\end{array}
\]
to obtain
\[
\begin{array}{rl}
 L(f^j(\sigma_0)) & \displaystyle= \frac{-2k(\frac{\sigma_{at}+z_j\sigma_{re}}{1+z_j})^2}{(\frac{-1+z_j}{1+z_j}\sqrt{\Delta})^2-\Delta}
=\displaystyle\frac{-2k(\frac{\sigma_{at}+z_j\sigma_{re}}{1+z_j})^2}{\frac{-\Delta}{(1+z_j)^2}4z_j}\\
& \displaystyle =\frac{-k(\sigma_{at}+z_j\sigma_{re})^2}{(-\Delta)2z_j} =\frac{-k(\sigma_{at}^2+z_j^2\sigma_{re}^2+2\sigma_{at}\sigma_{re}z_j)}{(-\Delta)2z_j}
\end{array}
\]
We have that $\sum_{j=1}^{j=b} z_j=\sum_{j=1}^{j=b} \frac 1{z_j}=0$.
Therefore
\[
 -\frac12T=\sum_{j=1}^{j=b} L(f^j(\sigma_0))=
 \sum_{j=1}^{j=b}
 \frac{-k(\sigma_{at}\sigma_{re})}{-\Delta}
 =b\frac{-k(\sigma_{at}\sigma_{re})}{-\Delta}
\]
Since $\sigma_{at}\sigma_{re}=\mu$ we conclude
\[
 T=\frac{2kb\mu}{-\Delta}
\]
as we wanted to prove.
\end{proof}


\section{Tachyonic collisions}
\label{secTachCol}
In the next result we consider \textit{full solutions}, \textit{i.e.} defined for all time, positive and negative.

\begin{thm}
\label{thmCondTach}
 For a mirror system of $N=4$ particles there is a tachyonic collision if and only if
 $0<\sigma_1(\sigma_1-2\E)$. Also:
 \begin{itemize}
  \item if $\Delta<0$ then every solution has infinitely many tachyonic collisions,
  \item if $\Delta=0$ then every solution has exactly two tachyonic collisions, which are consecutive,
  \item if $\Delta>0$ then a solution has  tachyonic collisions if and only if $(\sigma_1^0-\E)^2>\Delta$, and in this case the number of such collisions is two and they are consecutive,
 \end{itemize}
where $\sigma_1^0$ is the initial condition of the solution.
\end{thm}

\begin{proof}
By definition (recall \S\ref{secTachCol})
a collision is tachyonic provided that $\E^2-\P^2<0$.
Since $\sigma_1=E_1+P_1$ and $P_2=-E_2$, the tachyonic condition can be written as
$
 \E^2<(P_1+P_2)^2=(\sigma_1-\E)^2,
$
which is equivalent to
$0<\sigma_1(\sigma_1-2\E)$. This proves the first part.

 Suppose that  $\Delta<0$, \textit{i.e.} $\E^2<\mu$.
 We derive the auxiliary inequality
 \begin{equation}
  \label{ecuDeltaNegAux}
  \mu-2\sigma_1\E+\sigma_1^2>\E^2-2\sigma_1\E+\sigma_1^2=(\E-\sigma_1)^2\geq 0.
 \end{equation}
This implies that $(f(\sigma_1)-\sigma_1)\sigma_1>0$ because
\[
 (f(\sigma_1)-\sigma_1)\sigma_1
 =\left(\frac\mu{2\E-\sigma_1}-\sigma_1\right)\sigma_1
 =[\mu-\sigma_1(2\E-\sigma_1)]\frac{\sigma_1}{2\E-\sigma_1}
\]
Suppose that at $\sigma_1$ there is not a tachyonic collision, which means
\begin{equation}
\label{ecuNoTach}
0\geq\sigma_1(\sigma_1-2\E).
\end{equation}
This and Inequality \eqref{ecuDeltaNegAux} implies that
\[
 \sign [(f(\sigma_1)-\sigma_1)\sigma_1]=\sign [\mu-\sigma_1(2\E-\sigma_1)]>0.
\]
Now, if $\E>0$, from \eqref{ecuNoTach} we have $\sigma_1>0$ and
$f(\sigma_1)>\sigma_1$.
If we iterate $f$ to obtain a sequence $f^j(\sigma_1)$, $j\geq 1$,
we see that it increases (as $\Delta<0$ there is no fixed point)
until $f^{J}(\sigma_1)>2\E$ and at this moment there is a tachyonic collision.\footnote{It also holds that $f^{J+1}(\sigma_1)<0$ and there is one more tachyonic collision.}
The case $\E<0$ is analogous, considering that $\sigma_1<0$ and the corresponding sequence is decreasing.

 Suppose that $\Delta\geq 0$ and assume that $(\sigma_1^0-\E)^2>\Delta$ (this is always the case if
 $\Delta=0$).
 Since the interval between the fixed points is invariant by $f$ we have that
 $(\sigma_1^n-\E)^2>\Delta$ for all $n\in\Z$, \textit{i.e.} the iterates of $\sigma_1^0$ by $f$ are always outside that interval.
 If $\E>0$ and $\sigma_1^0$ is at the right of the repeller and close to it, then the sequence of its iterates is increasing and for some $J>0$ we have $f^J(\sigma_1^0)>2\E$, which is the first tachyonic collision. The second one is next because if $x>2\E$ then $f(x)<0$ ($x=f^J(\sigma_1^0$).
 The next iterates are positive (converging to the attractor) and there are no more tachyonic collisions.
\end{proof}

By Theorem \ref{thmCondTach} if $\Delta<0$ then the tachyonic collisions are frequent
and for $\Delta\geq 0$ they are quite special.
Thus, it seems interesting to give more details for the later case.
Define $\kappa=-2 E_2^0/\sigma_1^0$.

\begin{rmk}
For a solution with tachyonic collisions we have that $\kappa>0$.
To prove it, we have to show that $E_2^0\sigma_1^0<0$.
From Equation \eqref{ecuSimpliSigma}
we have that $2E_2^0=2\E-\sigma_1^0-\mu/\sigma_1^0$ and
$$2E_2^0\sigma_1^0=2\E\sigma_1^0-(\sigma_1^0)^2-\mu
=\Delta-(\E-\sigma_1^0)^2
$$
which is negative due to Theorem \ref{thmCondTach} and the presence of tachyonic collisions in the solution.
\end{rmk}

\begin{prop}
\label{propEscalaTach}
If $\Delta\geq -\E^2$ and the $n$-th collision is tachyonic then
$x_1^n/x_1^0\leq 4\kappa\E^2/\mu$.
\end{prop}

\begin{proof}
 Let $l(\sigma)=\frac{\sigma^2}{(\sigma-\E)^2-\Delta}$ for $\sigma\in\R$ satisfying
\begin{equation}
 \label{ecuTacho}
 \sigma(\sigma-2\E)>0.
\end{equation}
Let us show that the sign of its derivative is constant.
 We have $l'(\sigma)=\frac{-2\sigma[\E\sigma-\mu]}{[(\sigma-\E)^2-\Delta]^2}$
 and $\sign(l'(\sigma))=-\sign(\sigma[\E\sigma-\mu])$.
 Since $\Delta\geq -\E^2$, it is easy to prove that Inequality \eqref{ecuTacho} implies that $l'(\sigma)$ has constant sign.
Thus, to obtain the maximum value of $l$ we check the limits $\sigma=\pm\infty$, $\sigma=0$ and $\sigma=2\E$.
We have $l(\pm\infty)=1$, $l(0)=0$ and $l(2\E)=4\E^2/\mu$.
Since $\Delta\geq-\E^2$ we have
$2\E^2\geq\mu$ and the maximum of $l$ is at $\sigma=2\E$.
Finally, applying \eqref{ecuXnRecu} we have $x_1^n/x_1^0=\kappa l(\sigma_1^n)$.
\end{proof}

\subsection{An estimation for tachyonic collisions}
\label{secEstimation}
The Newtonian potential energy is $U=-Gm_1m_4/R$, where
$$G=6.7\times 10^{-11}\frac{m^3}{s^2kg}=\frac{6.7\times 10^{-11}}{(3.0\times 10^{8})^2}\frac{m}{kg}
=7.4\times 10^{-28}\frac{m}{kg}$$
in units of time so that the speed of light equals 1,
$R$ is the distance between the particles 1 and 4, \textit{i.e.} $R=-2x_1^n$ and $m_1m_4=\mu$.
That is
$$U=\frac{G\mu}{2x_1^n}.$$
Considering Proposition \ref{propSquareLaw}, in order to have a common limit value of the constant $G$ we will assume that $\Delta=0$ ($\E^2=\mu$).
From Lemma \ref{lemSuperRecurrencia}
we have
\[
 E_2^n=
 \frac{x_1^0E_2^0\sigma_1^n}{\sigma_1^0x_1^n}=\frac{k\sigma_1^n}{x_1^n}\sim \frac{k\E}{x_1^n}\text{ as $n\to\pm\infty$}.
\]
If we identify the gravitational energy $U$ and the energy of particles 2 and 3 we obtain
\[
 \frac{G\mu}{2x_1^n}=2 \frac{k\E}{x_1^n}\Rightarrow k=\frac{G\mu}{4\E}=\frac{G\E}4.
\]
Proposition \ref{propEscalaTach} gives us that $|x_1^n|\leq 8|k|=2G|\E|$
if the $n$-th collision is tachyonic.
If $m$ is the positive square root of $\mu$ then $|\E|=m$ (because $\Delta=0$).

Consider that particles 1 and 4 are \textit{heavy}, for instance they are neutrons, \textit{i.e.}
$m=1.7\times 10^{-27} kg$.
This gives us a bound of
\[
 2Gm=2\times
 7.4\times 10^{-28}\frac{m}{kg}
 1.7\times 10^{-27} kg= 2.5\times 10^{-54} \text{metres}
\]
for the distance between the neutrons when they enter into the tachyonic collisions.
That is, at this small distance their energies would turn negative for just two consecutive collisions.

\textit{Acknowledgements}. The author thanks the referees for several comments and suggestions that helped to improve the presentation of the article.


\begin{bibdiv}
\begin{biblist}

\bib{Art19}{article}{
author={A. Artigue},
title={Billiards and toy gravitons},
journal={Journal of Statistical Physics},
volume={175},
year={2019},
pages={213--232}}

\bib{BTT2007}{article}{
title={On the Zero Mass Limit of Tagged Particle Diffusion in the 1-d Rayleigh-Gas},
author={P. Bálint},
author={B. Tóth},
author={P. Tóth},
year={2007},
journal={Journal of Statistical Physics},
volume={127},
doi={10.1007/s10955-007-9304-2}}




\bib{Lap86}{article}{
author={I.R. Lapidus},
title={A useful notation for relativistic kinematics},
year={1986},
publisher={Am. J. Phys.},
volume={54},
doi={10.1119/1.14472}}




\end{biblist}
\end{bibdiv}
\end{document}